\documentclass[12pt]{article}
\usepackage{amsmath,amssymb,amsfonts,amsthm,amstext}
\usepackage{url,xspace}
\usepackage{color,graphicx}
\usepackage[hypertex]{hyperref}
\newcommand{\Z}{{\mathbb Z}}
\newcommand{\R}{{\mathbb R}}

\renewcommand{\P}{{\mathbb P}}
\newcommand{\E}{{\mathbb E}}
\newcommand{\leb}{{\mathcal L}}
\newcommand{\ind}{{\mathbf 1}}
\newcommand{\red}{{\cal R}}
\newcommand{\blue}{{\cal B}}
\newcommand{\rred}{[{\cal R}]}
\newcommand{\bblue}{[{\cal B}]}
\newcommand{\mat}{{\cal M}}
\newcommand{\mhat}{{\widehat\mat}}

\newcommand{\parend}{{\hfill $\Diamond$}\par\vskip2\parsep}
\newcommand{\dof}{\bf\boldmath}
\DeclareMathOperator{\Poi}{Poi}

\newtheorem{thm}{Theorem}
\newtheorem{lemma}[thm]{Lemma}

\newtheorem{cor}[thm]{Corollary}
\newtheorem{question}{Question}

\newtheorem{obs}[thm]{Observation}

\newcounter{mycount}
\newenvironment{mylist}{\begin{list}{{\rm (\roman{mycount})}}%
{\usecounter{mycount}\setlength{\itemsep}{-3pt}\setlength{\leftmargin}{12pt}}}
{\end{list}}

\title{Geometric Properties of Poisson Matchings}
\author{Alexander E. Holroyd}
\date{1 September 2009}

\begin{document}
\maketitle

{\centering\em Dedicated to Oded Schramm, 10 December 1961 -- 1
September 2008\\}

\renewcommand{\thefootnote}{}
\footnotetext{{\bf\hspace{-6mm}Key words:} Poisson process,
point process, matching} \footnotetext{{\bf\hspace{-6mm}AMS
2010 Mathematics Subject Classifications:} 60D05, 60G55, 05C70}
\footnotetext{\hspace{-6mm}Funded in part by Microsoft and
NSERC}
\renewcommand{\thefootnote}{\arabic{footnote}}

\begin{abstract}
Suppose that red and blue points occur as independent Poisson
processes of equal intensity in $\R^d$, and that the red points
are matched to the blue points via straight edges in a
translation-invariant way.  We address several closely related
properties of such matchings. We prove that there exist matchings
that locally minimize total edge length in
$d=1$ and $d\geq 3$, but not in the strip $\R\times[0,1]$. We
prove that there exist matchings in which every bounded set
intersects only finitely many edges in $d\geq 2$, but not in
$d=1$ or in the strip.  It is unknown whether there exists a
matching with no crossings in $d=2$, but we prove positive
answers to various relaxations of this question. Several open
problems are presented.
\end{abstract}

\section{Introduction}
%%%%%%%%%%%%%%%%%%%%%%%
\label{intro}

Let $\red$ and $\blue$ be simple point processes in $\R^d$. The
{\dof support} of a point process $\Pi$ is the random set
$[\Pi]:=\{x: \Pi(\{x\})=1\}$; the elements of $[\Pi]$ are
called {\dof $\Pi$-points}.  We call $\red$-points {\dof red
points} and $\blue$-points {\dof blue points}.  A (perfect,
two-color) {\dof matching scheme} of $\red$ and $\blue$ is a
simple point process $\mat$ in $(\R^d)^2$
%, on some shared probability space with $\red$ and $\blue$,%
 such that almost surely
$(V,E)=([\red]\cup[\blue],[\mat])$ is a perfect matching of
$\rred$ to $\bblue$ (i.e.\ a bipartite graph with vertex
classes $\rred,\bblue$ and all degrees 1).  We call
$\mat$-points {\dof edges}. Similarly, we say that $\mat$ is a
{\dof partial} matching scheme if all degrees are at most 1;
and $\mat$ is a {\dof one-color} matching scheme of $\red$ if
$([\red],[\mat])$ is a.s.\ a simple undirected graph will all
degrees 1. A matching scheme $\mat$ is {\dof
translation-invariant} if the law of $(\red,\blue,\mat)$ is
invariant under all translations of $\R^d$.  We also consider
matchings of point processes on the {\dof strip}
$\R\times[0,1)$, in which case translation-invariance refers to
all translations in the first coordinate direction.

In $\R^2$ or the strip, we call a matching scheme $\mat$ {\dof
planar} if a.s.\ for any distinct matched pairs
$(r,b),(r',b')\in[\mat]$, the two closed line segments joining
$r$ to $b$ and $r'$ to $b'$ do not intersect.

We focus on the case where $\red$ and $\blue$ are independent
Poisson processes of equal intensity.  It is unknown whether
there exists a translation-invariant planar matching scheme in
$\R^2$ (see the discussion on open problems below), but the
following natural variations of this question all lead to
positive answers.

\begin{thm}[Planarity]\label{relax}
Let $\red$ and $\blue$ be independent Poisson processes of
intensity 1 in $\R^2$. The following all exist.
\begin{mylist}
\item A planar matching scheme of $\red$ to $\blue$
    that is {\em not} translation-invariant.
\item A translation-invariant planar one-color matching scheme
    of $\red$.
\item A translation-invariant planar {\em partial} matching
    scheme of $\red'$ to $\blue$, in which every blue point is matched,
    where $\red',\blue$ are independent Poisson processes of
    intensities $\lambda,1$, for any $\lambda>1$.
\item A translation-invariant process consisting of a matching
    of $\red$ to $\blue$, together with a polygonal arc in $\R^2$ joining
    each pair of matched points, such that the arcs do not
    intersect.
\end{mylist}
Furthermore, each of (i)--(iv) exists on the strip.
\end{thm}

The following concept has close connections with planarity. We
call a matching scheme $\mat$ {\dof minimal} if a.s.\ every
finite set of edges is matched in a way that minimizes total
edge length, i.e.\ for any
$\{(r_i,b_i)\}_{i=1,\ldots,n}\subset[\mat]$ (where
$r_i\in\rred$ and $b_i\in\bblue$ for each $i$), we have
$$\sum_i |r_i-b_i| = \min_{\sigma} \sum_i |r_i-b_{\sigma(i)}|,$$
where the minimum is over all permutations $\sigma$ of
$1,\ldots,n$, and $|\cdot|$ denotes the Euclidean norm.  An
elementary argument (see the discussion below) shows that any
minimal matching scheme in $\R^2$ is planar.  However, the
concept of minimality is natural in all dimensions, and we have
the following surprising facts.

\begin{thm}[Minimality]\label{minimal} Let $\red$ and $\blue$
be independent Poisson processes of intensity 1.
\begin{mylist}
\item There exists a translation-invariant minimal matching of
    $\red$ to $\blue$ in $\R^d$ for $d=1$ and all $d\geq 3$.
\item There does not exist a translation-invariant minimal
    matching of $\red$ to $\blue$ in the strip.
\end{mylist}
\end{thm}
The remaining case of $\R^2$ is open.

Next, we say that an edge $(r,b)\in[\mat]$ {\dof crosses} a set
$S\subset\R^d$ if the closed line segment from $r$ to $b$
intersects $S$.  We call a matching scheme $\mat$ {\dof locally
finite} if a.s., every bounded set $S\subset\R^d$ is crossed by
only finitely many edges.

\begin{thm}[Local finiteness]\label{locfin}Let $\red$ and $\blue$
be independent Poisson processes of intensity 1.
\begin{mylist}\sloppy
\item There does not exist a translation-invariant locally
    finite matching scheme in $\R$, nor in the strip.
\item There exist translation-invariant locally
    finite matching schemes in $\R^d$ for all $d\geq 2$.
\end{mylist}
\end{thm}

Finally, we establish the following conditional result.
\begin{thm}[Minimal implies locally finite]\label{cond}
Let $\red$ and $\blue$ be independent Poisson processes of
intensity 1 in $\R^d$, where $d\geq 2$.  Any translation-invariant minimal
matching scheme must be locally finite.
\end{thm}

Note that Theorems \ref{minimal}(i) and \ref{locfin}(i) show
that the assertion of Theorem \ref{cond} fails in $d=1$.

\subsubsection*{Motivation and open problems}

Invariant Poisson matching schemes, and particularly their
quantitative properties, were studied extensively in
\cite{hpps}.  Other related work appears in
\cite{deijfen,h-p-trees,krikun,soo,timar}.  The present work is
largely motivated by the following question, which was posed by
Yuval Peres in 2002, stated in \cite{hpps}, and remains open.

\begin{question}\label{theq}
For $\red$ and $\blue$ independent Poisson processes of
intensity 1 in $\R^2$, does there exist a translation-invariant
planar matching scheme?
\end{question}

It is far from clear what answer to guess to the above
question, with several of the results presented here perhaps
suggesting opposite answers.  We propose the following natural
variants.

\begin{question}\label{stripq}
For $\red$ and $\blue$ independent Poisson processes of
intensity 1 in the strip, does there exist a
translation-invariant planar matching scheme?
\end{question}

\begin{question}\label{minq}
For $\red$ and $\blue$ independent Poisson processes of
intensity 1 in $\R^2$, does there exist a translation-invariant
minimal matching scheme?
\end{question}

Note on the other hand that Theorem \ref{minimal}(ii) gives a
negative answer to the analogous question of minimal matchings
in the strip. The following simple consequence of the triangle
inequality implies immediately that a minimal matching in
$\R^2$ is planar, so a positive answer to Question \ref{minq}
would imply a positive answer to Question \ref{theq}.  (A
positive answer to Question \ref{stripq} would also imply a
positive answer to Question \ref{theq}, by a simple argument --
see Section \ref{strip-line}).  We say that a set of points
$K\subset\R^d$ is {\dof parallel-free} if there do not exist
$x,y,u,v\in K$ and $a\in\R$ with $\{x,y\}\neq\{u,v\}$ and
$x-y=a(u-v)\neq 0$.

\begin{obs}[Finite minimum matchings are planar]\label{obs}
Let $R,B\in\R^2$ be disjoint finite sets of equal cardinality,
and suppose $R\cup B$ is parallel-free.
Then in any perfect matching of $R$ to $B$
that minimizes the total length, the line segments joining
matched pairs do not intersect.
\end{obs}

In the light of this observation, the following possible
approach to constructing a translation-invariant planar
matching seems natural.  Take $n$ red and $n$ blue points
uniformly at random in a square of area $n$, randomly
translated so that the origin is uniformly distributed in the
square. Consider the matching of minimum total length, and take
suitable a limit in distribution as $n\to\infty$. For such an
approach to be successful, the limit must be a genuine matching
-- it is possible that instead the partner of a point goes to
infinity. If it exists, the limiting matching would be minimal.
This motivates Question \ref{minq}.  We will employ a somewhat
similar limiting argument in the proof of Theorem
\ref{minimal}(i) in $d\geq 3$.

Question \ref{minq} remains open if the
invariance requirement is dropped.
\begin{question}\label{noninv}
For $\red$ and $\blue$ independent Poisson processes of
intensity 1 in $\R^2$, does there exist a minimal matching
scheme?
\end{question}

The notion of locally finite matching (see Theorem
\ref{locfin}) becomes particularly interesting in $\R^2$, owing
to the following result proved in \cite[Proof of Theorem 2,
 case $d=2$]{hpps}.
\begin{thm}[Infinite mean crossings; \cite{hpps}]\label{expected}
Let $\red$ and $\blue$ be independent Poisson processes of
intensity 1 in $\R^2$.  In any translation-invariant matching
scheme, for any fixed bounded set $S\subset\R^2$, the number of
edges that cross $S$ has infinite expectation.
\end{thm}

Notwithstanding Theorem \ref{locfin}(ii), one might speculate
that if a {\em minimal} matching scheme exists in $\R^2$, the
infinite expectation in Proposition \ref{expected} should in
some sense be ``spread around evenly'', so that the matching is
not locally finite.  Combined with Theorem \ref{cond}, this
perhaps suggests a negative answer to Question \ref{minq}.

Returning to the issue of planarity, we propose the following
question.
\begin{question}
Do there exist jointly ergodic point processes $\red$ and
$\blue$ in $\R^2$, both of intensity 1, for which there is
(provably) no planar translation-invariant matching scheme?
\end{question}

Finally, note that our definition of a matching scheme requires
only that $\red,\blue,\mat$ are all defined on some joint
probability space, so the matching may involve additional
randomization besides that of the red and blue processes.  If
instead $\mat=f(\red,\blue)$ for some deterministic function
$f$, the matching scheme is called a {\dof factor}.  See e.g.\
\cite{hpps} for more on this distinction.  Most of the matching
schemes we construct will not be factors.  Another interesting
line of enquiry (which we do not pursue here) is to determine
whether there exist factor matching schemes satisfying the
various conditions under consideration.

\subsubsection*{Some notation}

We write $\leb$ for Lebesgue measure on $\R^d$, and $|\cdot|$
for the Euclidean norm.  The ball is denoted
$B(r):=\{x\in\R^d:|x|<r\}$.  If $\mat$ is a matching scheme, we
write $\mat(x)$ for the {\dof partner} of a red or blue point
$x$, i.e.\ the unique point such that $(x,\mat(x))\in[\mat]$.
Similarly in a deterministic matching $m$ we write $m(x)$ for
the partner of $x$.

\section{Finite matchings}
%%%%%%%%%%%%%%%%%%%%%%%%%%

In this section we verify some elementary facts.

\begin{proof}[Proof of Observation \ref{obs}]\sloppy
Suppose on the contrary that, in some length-minimizing
matching, two edges intersect.  By the parallel-free
assumption, they must intersect non-trivially, i.e.\  in a
single point that is not one of their endpoints. But now an
application of the triangle inequality shows that the other
possible matching of these four points has strictly smaller
total length; see Figure \ref{uncross}.
\end{proof}
\begin{figure}
\centering
\begin{picture}(0,0)(0,0)
\put(5,15){$r$}
\put(105,5){$r'$}
\put(75,110){$b'$}
\put(120,50){$b$}
\put(77,46){$x$}
\end{picture}
\resizebox{1.6in}{!}{\includegraphics{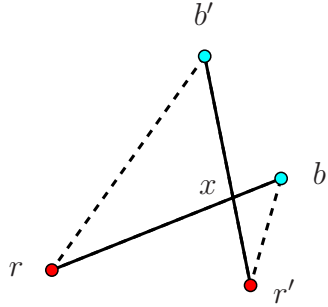}}
\caption{Uncrossing a pair of edges: the triangle inequality gives
 $rb'+r'b<rx+xb'+r'x+xb=rb+r'b'$.}
\label{uncross}
\end{figure}

\begin{lemma}\label{parallel}
In a homogeneous Poisson process $\Pi$ in $\R^d$ with $d\geq
2$, $[\Pi]$ is a.s.\ parallel-free.
\end{lemma}

\begin{proof}
It is enough to check this for the Poisson process restricted
to a ball, and in this case we may also condition on the number
of points in the ball.  So it suffices to check that if
$x_1,\ldots,x_4$ are independent uniform points in a fixed
ball, then a.s.\ the vectors $x_1-x_2$, $x_1-x_3$ and $x_3-x_4$
are pairwise non-parallel, which is elementary.
\end{proof}

As remarked earlier we can deduce the following.
\begin{cor}[Minimal implies planar]
For $\red$ and $\blue$ independent Poisson processes of
intensity 1 in $\R^2$, any minimal matching scheme (if such
exists) is planar.
\end{cor}

\begin{proof}
This is immediate from Observation \ref{obs} and Lemma
\ref{parallel}; in fact we need the minimality property only
for sets of two edges.
\end{proof}

In order to use Observation \ref{obs}, we will often wish to
consider the perfect matching of minimum total length between
two finite sets of points $R,B\subset\R^d$.  If $d\geq 2$ and
the points are any subset of the points of a Poisson process,
one may show that such a minimum matching is a.s.\ unique.
Formally, this fact is not needed, because we can choose among
minimum-length matchings according to some fixed rule, such as
the earliest in lexicographic order with respect to the
coordinates of the points.

\section{Matchings in strips and lines}
%%%%%%%%%%%%%%%%%%%%%%%%%%%%%%%%%%%%%%%
\label{strip-line}

In this section we prove Theorem \ref{relax}.  Each part will
be proved in the strip, and the case of $\R^2$ then follows as
an easy consequence (see below).  We also prove Theorem
\ref{minimal}(i) in the case $d=1$.

\begin{proof}[Proof of Theorem \ref{relax}, $\R^2$ case]
Divide $\R^2$ into the disjoint strips
$$\R\times [i,i+1), \quad i\in\Z.$$
Within each strip, take an independent copy of the matching
scheme (together with the associated red and blue points) on
the strip from the appropriate part (i)--(iv) of the theorem
(see the proofs below). This yields a matching scheme $\mat$ in
$\R^2$ which inherits the appropriate properties of the
matching on the strip and is invariant under translations in
$\R\times\Z$ (for (ii)--(iv)).  To achieve full
translation-invariance in $\R^2$ (for (ii)--(iv)), let $U$ be a
uniform random variable in $[0,1)$, independent of
$(\red,\blue,\mat)$, and translate the entire process
$(\red,\blue,\mat)$ by the vector $(0,U)$.
\end{proof}

\paragraph{Remark.}  By the above argument, a positive answer to
Question \ref{stripq} would imply a positive answer to Question
\ref{theq}.\parend

We will make frequent use of the following object. Given
$\red,\blue$ in the strip, define a function $F:\R\to\Z$ by
\begin{equation}\label{rw}
\begin{split}
F(0)&=0;\\
F(y)-F(x)&=(\red-\blue)\big((x,y]\times[0,1)\big),\quad x<y.
\end{split}\end{equation}
Thus $F$ is a right-continuous continuous-time simple symmetric
random walk on the integers, with its up-steps and down-steps
corresponding to red and blue points respectively.  See Figure
\ref{arc}.

\begin{proof}[Proof of Theorem \ref{relax}(i), strip case]
With $F$ as in \eqref{rw}, define
$$Z:=\{z\in\R: F(z)=0 \text{ and }F(z-)\neq 0\};$$
i.e.\ the set of left endpoints of the intervals where $F$ is
zero.  Almost surely, $Z$ is a discrete set, and because $F$ is
a recurrent random walk, $Z$ is unbounded in both the positive
and negative directions.  If $z_1<z_2$ are two consecutive
elements of $Z$, then the rectangle $(z_1,z_2]\times[0,1)$
contains equal numbers of red and blue points.  Therefore,
within each such rectangle, take the matching of minimum total
edge length, and appeal to Observation \ref{obs} and Lemma
\ref{parallel}.
\end{proof}

\paragraph{Remark.}  The above construction cannot be adapted
to give a translation-invariant matching on the strip, because
the random walk $F$ is null-recurrent, therefore $Z$ has zero
density.\parend

\begin{proof}[Proof of Theorem \ref{relax}(ii), strip case]
Let $(r_i)_{i\in\Z}=((x_i,y_i))_{i\in\Z}$ be the points of
$[\red]$, ordered so that their first coordinates are in
increasing order (i.e.\ $x_i<x_{i+1}\;\forall i$), and so that
$r_0$ is the first point to the right of the origin (so
$x_{-1}<0<x_0$). Conditional on $\red$, choose one of the two
matchings
$$\{\ldots,(r_{-2},r_{-1}),(r_0,r_1),(r_2,r_3),\ldots\}; \quad
\{\ldots,(r_{-1},r_{0}),(r_1,r_2),(r_3,r_4),\ldots\}; $$
each with probability $1/2$.
\end{proof}

\begin{proof}[Proof of Theorem \ref{relax}(iii), strip case]
Let $F$ be as in \eqref{rw}, but with $\red'$ in place of
$\red$.  Now $F$ is a biased random walk with positive drift.
Hence the set of {\dof cut-times} given by
$$C:=\Big\{x\in\R: \sup_{t<x} F(t)=F(x-)<F(x)=\inf_{t\geq x} F(t)\Big\}$$
forms a translation-invariant ergodic point process of positive
intensity in $\R$.  If $c_1<c_2$ are two consecutive cut-times,
then the rectangle $(c_1,c_2]\times[0,1)$ contains strictly
more red than blue points.  Therefore, within each such
rectangle, take the matching which minimizes the total edge
length from among all possible partial red-blue matchings of
maximum cardinality (i.e.\ those in which all blue points are
matched).  By Observation \ref{obs} and Lemma \ref{parallel},
the resulting matching scheme has the required properties.
\end{proof}

\begin{figure}
\centering
\resizebox{4.5in}{!}{\includegraphics{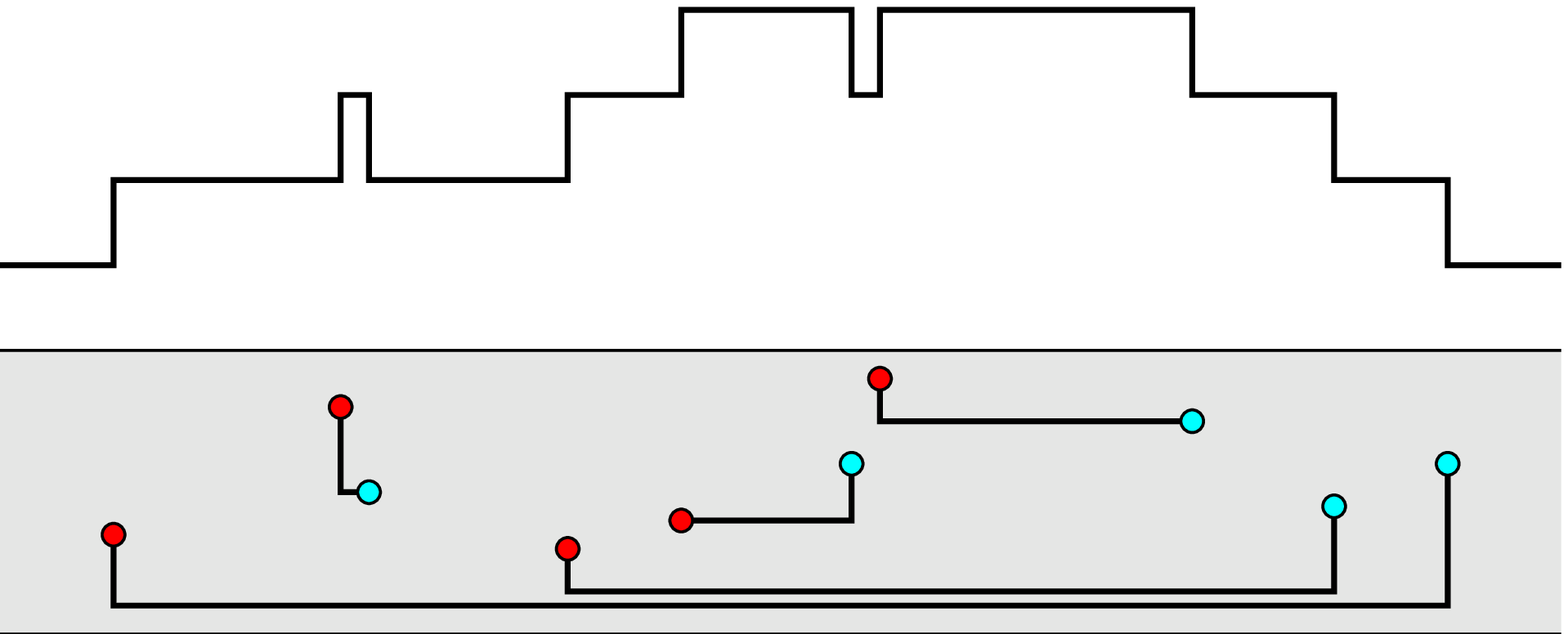}}
\caption{The random walk $F$ associated with red and blue points in the strip,
and the matching with disjoint arcs from the proof of Theorem \ref{relax}(iv).}
\label{arc}
\end{figure}
\begin{proof}[Proof of Theorem \ref{relax}(iv), strip case]
We first construct the matching scheme $\mat$.  Given
$\red,\blue$ on the strip, let $F$ be as in \eqref{rw}. Suppose
$r=(r_1,r_2)\in[\red]$ is a red point, so that
$F(r_1)=F(r_1-)+1$. We match $r$ to the blue point
$b=(b_1,b_2)$, where
$$b_1:=\inf\{t>r_1: F(t)=F(r_1-)\}.$$
Thus, $b$ marks the end of $F$'s upward excursion starting at
$r$, or equivalently $b$ is the first point to the right of $r$
such that the red and blue points in the intervening rectangle
equalize.  It is easy to check that $\mat$ is
indeed a translation-invariant matching scheme.  Furthermore,
if $(r,b),(r',b')$ are two distinct edges of the matching then
the intervals $[r_1,b_1]$ and $[r_1',b_1']$ are either disjoint
or nested one inside the other.

Now we construct the polygonal arcs.  If $r$ and $b$ are two
matched points, we will join them by a polygonal arc with
vertices
$$r=(r_1,r_2),(r_1,H),(b_1,H),(b_1,b_2)=b;$$
we need
only choose $H$ smaller (say) than the heights of all the
intervening arcs. This is achieved by taking for example
$H:=L/D$, where
$$L=\min\Big\{y:(x,y)\in
([\red]\cup[\blue])\cap ([r_1,b_1]\times[0,1))\Big\}$$
is the height of the lowest point between $r$ and $b$
(including $r$ and $b$), and
$$D=\max_{t\in [r_1,b_1]}F(t)-F(r_1-)$$
is the maximum nesting depth between $r$ and $b$.
\end{proof}

\paragraph{Remark.}
The matching $\mat$ constructed in the above proof has an
interpretation as the unique Gale-Shapley stable matching with
preferences based on one-sided horizonal distance, and also as
a version of a matching introduced by Meshalkin in the
construction of finitary isomorphisms.  See
\cite{gale-shapley,h-p-trees,meshalkin} for details.  The same
matching is used our next proof.\parend

\begin{proof}[Proof of Theorem \ref{minimal}(i), case $d=1$]
Given $\red$ and $\blue$ Possion processes on $\R$, define $F$
as in \eqref{rw} except replacing the rectangle
$(x,y]\times[0,1)$ with the interval $(x,y]$, so $F$ is again a
simple symmetric random walk.  Construct a matching scheme
exactly as in the proof of Theorem \ref{relax}(iv) by matching
the red point $r$ to the blue point
$$b:=\inf\{t>r: F(t)=F(r-)\}.$$

We claim that this matching scheme $\mat$ is minimal.  To prove
this, let $(r_1,b_1),\ldots,(r_n,b_n)\in[\mat]$ be any finite
set of edges of the matching -- we must prove that no other
matching of these red and blue points has smaller total length.
Since any two edges of $\mat$ span disjoint or nested
intervals, there exists a bounded interval $(u,v)\subset\R^d$
containing $\{r_i,b_i\}_{i=1,\ldots ,n}$, and such that every
point in $(u,v)$ has its partner in $(u,v)$.  (To prove this,
let $I$ be the smallest interval containing the original
points, then let $(u,v)$ be the smallest interval containing
all the points in $I$ and their partners).  Therefore, it
suffices to prove the above minimality statement under the
assumption that $\{r_i,b_i\}_{i=1,\ldots ,n}$ are the {\em
only} red and blue points in $(u,v)$.

For any perfect matching $m$ of $\{r_i\}_{i=1,\ldots ,n}$ to
$\{b_i\}_{i=1,\ldots ,n}$, and any $t\in(u,v)$, define
$$h_m(t):=\#\big\{i:\; r_i\leq t\leq m(r_i)
\text{ or } m(r_i)\leq t\leq r_i\big\};$$
(i.e.\ the number of edges that cross $t$).  Note that the
total edge-length of the matching $m$ may be expressed thus:
\begin{equation}\label{integral}
\sum_i |r_i-m(r_i)| =\int_u^v h_m(t)\,dt.
\end{equation}

We claim that, writing $M=\{(r_i,b_i)\}_{i=1,\ldots ,n}$ for
the restriction of $\mat$ to these points, we have $h_M(t)\leq
h_m(t)$ for all $t\in(u,v)\setminus\{r_i,b_i\}_{i=1,\ldots,n}$.
Once this is proved, the required minimality follows from
\eqref{integral}.  To prove this inequality, first note that
for all such $t$,
$$h_M(t)=F(t)-F(u).$$
Indeed, this holds for any matching in which every edge has the
red point to the left of the blue point: as $t$ increases from
$u$ to $v$, the quantity $h_M(t)$ increases by 1 at each red
point, and decreases by 1 at each blue point, so it equals the
right side. On the other hand, for any $m$ and all such $t$,
$$h_m(t)\geq F(t)-F(u),$$
because the right side equals the excess of red points minus blue
points in $(u,t]$, so at least this many edges must cross $t$.
\end{proof}

\section{Impossibility results}
%%%%%%%%%%%%%%%%%%%%%%%%%%%%%%%%%%%%%%%%%%%%%%%%%%%%%%%%%%%

In this section we prove Theorem \ref{locfin}(i), Theorem
\ref{minimal}(ii), and Theorem \ref{cond}.

\begin{proof}[Proof of Theorem \ref{locfin}(i)]
The proofs for the strip and the line are nearly identical.  We
first consider the strip.

Let $\mat$ be a translation-invariant matching scheme on the
strip. Assume without loss of generality that the matching is
ergodic under the group of translations of $\R$ (otherwise
consider its ergodic components).  We will prove that a.s.\
infinitely many edges cross the line segment
$\{0\}\times[0,1)$.  Suppose on the contrary that for some
$k<\infty$, exactly $k$ edges cross $\{0\}\times[0,1)$ with
positive probability. Define the random set
\begin{align*}
S:=\big\{x\in\R:\; &\text{$\{x\}\times[0,1)$ does not intersect
$[\red]\cup[\blue]$,} \\
&\text{and is crossed by exactly $k$
edges}\big\}.
\end{align*}
Then the above assumptions imply that $S$ is a
translation-invariant ergodic random set of positive intensity,
say $\lambda$, and thus
\begin{equation}\label{erg}
\lim_{n\to\infty} \frac{\leb (S\cap[0,n))}{n}=\lambda
\quad\text{ a.s.}
\end{equation}

If $s<t$ are any two elements of $S$, then there are at most
$2k$ edges from the rectangle $[s,t)\times[0,1)$ to its
complement, therefore the difference between the numbers of red
and blue points in this rectangle is at most $2k$.  Hence, with
$F$ defined as in \eqref{rw}, $|F(s)-F(t)|\leq 2k$. This
implies that there exists some random integer $H$ such that
a.s.,
$$F(s)\in [H,H+2k] \text{ for all } s\in S.$$
However, since $F$ is a null-recurrent random walk, we have for
every integer $h$,
$$\lim_{n\to\infty} \frac{\leb \{s\in\R: F(s)\in[h,h+2k]\}}{n}=0 \quad\text{ a.s.,}$$
giving a contradiction to \eqref{erg}, and completing the
proof.

The proof in the case of $\R$ is identical except that we
consider edges crossing the site $x$ rather than the line
segment $\{x\}\times[0,1)$.
\end{proof}

\begin{proof}[Proof of Theorem \ref{cond}]
Let $d\geq 2$ and let $\mat$ be a matching in $\R^d$ that is
not locally finite. On the event that some bounded set is
crossed by infinitely many edges, we will derive a
contradiction to minimality. Suppose that $S\subset\R^d$ is a
bounded set that is crossed by infinitely many edges.  Suppose
also that $[\red]\cup[\blue]$ is parallel-free, and locally finite as a subset of $\R^d$ -- both of these properties hold a.s.\ for a
Poisson process (Lemma \ref{parallel}). The remainder of the proof will be a
deterministic geometric argument given these assumptions.

Let $L:=\{\{a u:a\in\R\}:u\in\R^d\setminus\{0\}\}$
be the projective space of all lines passing through the
origin; $L$ is a compact metric space under the angle metric.
The {\dof direction} of an edge $(r,b)$ is the line
$\{a(r-b):a\in\R\}\in L$.  Since $L$ is compact,
the set of directions of all edges that cross $S$ has an
accumulation point; fix $\ell\in L$ to be one such. By the
local finiteness of $[\red]\cup[\blue]$, the set of edges that
intersect $S$ and have both endpoints outside any given bounded
set contains a sequence whose directions converge to
$\ell$.

Now fix some edge $(r,b)\in[\mat]$ whose direction is not equal
to $\ell$ (this is possible, otherwise all edges would be
parallel), and let $\theta$ be the acute angle between $(r,b)$
and $\ell$.  By the above observations, for any $t>0$ there
exists an edge $(r',b')\in[\mat]$ that crosses $S$, makes angle
less than $\theta/2$ with $\ell$, and has both endpoints
$r',b'$ outside $B(t)$.  We claim that if $t$ is sufficiently
large, any such edge satisfies
$$|r-b'|+|r'-b|<|r-b|+|r'-b'|,$$
contradicting minimality.

Figure \ref{rematch} illustrates the proof of this claim: take
any (doubly infinite) line $\Lambda$ intersecting $S$ and
making angle less than $\theta/2$ with $\ell$. If $r'$ and $b'$
are points on $\Lambda$ going to infinity in opposite
directions, such that both are at distance $t$ from $O$, where
$t\to\infty$, then the difference $|r-b'|+|r'-b|-|r'-b'|$
converges to the orthogonal projection of $(r,b)$ onto
$(r',b')$.  The latter equals $\pm |r-b|\cos\alpha$, where
$\alpha$ is the angle between $(r,b)$ and $\Lambda$ (and the
sign depends on the order of $r'$ and $b'$). Furthermore, this
convergence is uniform in the choice of the line $\Lambda$,
because the distance between $\Lambda$ and $(r,b)$ is bounded
for lines that intersect $S$. Since $\alpha>\theta/2$ we have
$\pm |r-b|\cos\alpha \leq |r-b|\cos (\theta/2)<|r-b|$, and the
above claim follows.
\begin{figure}
\centering
\begin{picture}(0,0)(0,0)
\put(20,20){$r'$}
\put(345,20){$b'$}
\put(205,33){$b$}
\put(150,60){$r$}
\end{picture}
\resizebox{5in}{!}{\includegraphics{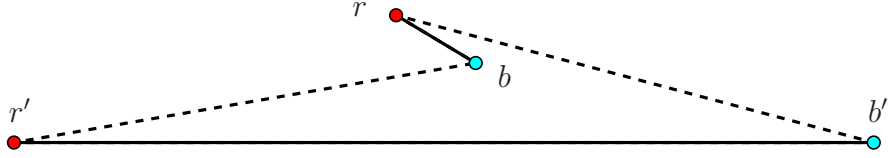}}
\caption{Contradicting minimality: as $r'$ and $b'$ go to infinity along a fixed line,
the difference between the total dashed length and
$|r'-b'|$ converges to the projection of $(r,b)$ on $(r',b')$.}
\label{rematch}
\end{figure}
\end{proof}

\begin{proof}[Proof of Theorem \ref{minimal}(ii)]
The proof of Theorem \ref{cond} above applies unchanged on the
strip (the direction $\ell$ must of course be horizonal), and
shows that any minimal matching scheme must be locally finite.
On the other hand, Theorem \ref{locfin}(i) states that a
locally finite translation-invariant matching scheme cannot
exist on the strip.
\end{proof}

\section{Locally finite matching}
%%%%%%%%%%%%%%%%%%%%%%%%%%%%%%%%%

In this section we prove Theorem \ref{locfin}(ii).

\begin{lemma}
\label{poi} Let $X,X',Y$ be independent Poisson random
variables with respective means $\lambda,\lambda,\mu$, where
$\mu\leq \lambda$.  Then
$$\P\big(X-X'\geq Y\big)\leq  \exp -\frac{\mu^2}{6\lambda}.$$
\end{lemma}

\begin{proof}
Write $Z:=Y+X'$ and $\nu:=\lambda+\mu$, so that $Z$ is Poisson
with mean $\nu$ and independent of $X$, and $X-X'\geq Y$ is
equivalent to $X\geq Z$.   Now we apply a Chernoff bound:
taking $s=\sqrt{\nu/\lambda}$ we have
\begin{align*}
\P(X\geq Z)&=\P(s^{X-Z}\geq 1)
\leq \E s^{X-Z} \\
&=\exp \big[\lambda (s-1)+\nu(s^{-1} -1)\big]
=\exp \big[2\sqrt{\lambda\nu}-\lambda-\nu\big].
\end{align*}
Writing $\delta=\mu/\lambda$, the last expression equals
\begin{align*}
&\exp 2\lambda \big[\sqrt{1+\delta}-1-\delta/2\big]\\
{ }\leq&\exp 2\lambda \Big[-\frac{\delta^2}{12}\Big]
= \exp -\frac{\mu^2}{6\lambda},
\end{align*}
(where we used the fact that $\sqrt{1+\delta}\leq
1+\delta/2-\delta^2/12$ for $\delta\in[0,1]$).
\end{proof}

\begin{proof}[Proof of Theorem \ref{locfin}(ii)]
We first argue that it is suffices to construct a locally
finite matching scheme in the case $d=2$.  Starting from such a
scheme, we may obtain a matching scheme in the slab
$\R^2\times[0,1)^{d-2}$ (where $d\geq 3$) by assigning each red
or blue point $x\in\R^2$ a location $(x,U)$, where the $U$ are
independent and uniformly random in $[0,1)^{d-2}$.  Now take
independent copies of this matching in each of the slabs
$\R^2\times(z+[0,1)^{d-2})$, for $z\in\Z^{d-2}$; the resulting
matching in $\R^d$ clearly inherits the locally finite
property, and is invariant under all translations in $\R^2
\times\Z^{d-2}$.  To obtain a fully translation-invariant
version, translate by a uniformly random element of
$\{0\}^2\times [0,1)^{d-2}$.

Similarly, it now suffices to find a matching scheme in $\R^2$
that is invariant under translations of $\Z^2$; then we obtain
a fully translation-invariant version by applying a translation
by a uniformly random element of $[0,1)^2$.

We start by defining a random sequence of successively coarser
partitions of $\R^2$ into rectangles.  Let $a_n=n!$.  For each
$n=1,2\ldots$, an {\dof $n$-block} will be an
$a_n$-by-$a_{n-1}$ (respectively $a_{n-1}$-by-$a_{n}$)
rectangle if $n$ is even (respectively odd).  Each $n$-block
$A$ will be a disjoint union of $a_n/a_{n-2}$ $(n-1)$-blocks,
called the {\dof children} of $A$.  The left-most (respectively
bottom-most) child is called the {\dof heir} of $A$.  See
Figure \ref{lead}.  We choose the blocks in a $\Z^2$-invariant
way as follows.  The $1$-blocks are all the squares $z+[0,1)^2$
for $z\in\Z^2$.  Let $r_n$ be a uniformly random integer in
$[0,a_n/a_{n-2})$, where the $r_n$ are independent of each
other and of $\red,\blue$. Let $t_n=r_n a_{n-2}+r_{n-2}
a_{n-4}+\ldots$, where the last term in this sum is $r_2a_0$
(respectively $r_3a_1$),  and define an $n$-block to be any
rectangle of the form $[x a_n+t_n, (x+1)a_n+t_n)\times[y
a_{n-1}+t_{n-1},(y+1)a_{n-1}+t_{n-1})$ where $(x,y)\in\Z^2$ if
$n$ is even (respectively, the same with the coordinates
reversed if $n$ is odd).
\begin{figure}
\centering
\begin{picture}(0,0)(0,0)
\put(35,20){$C$}
\put(35,60){$B$}
\put(30,-5){$a_{n-2}$}
\put(-17,20){$a_{n-3}$}
\put(145,128){$a_{n}$}
\put(290,60){$a_{n-1}$}
\end{picture}
\resizebox{4in}{!}{\includegraphics{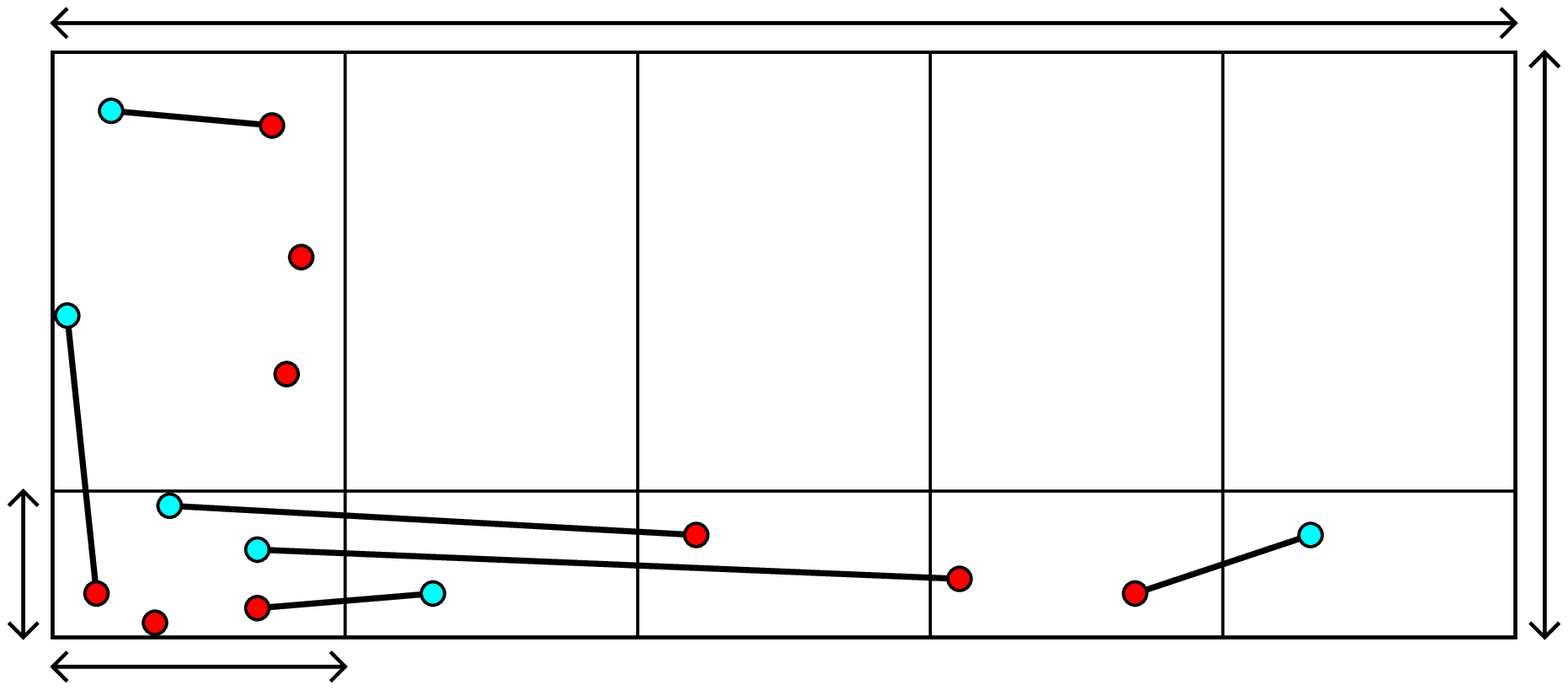}}
\caption{An $n$-block $A$,
its children including its heir $B$, and their heirs including $B$'s heir $C$.
Also shown are the new edges added in stage $n$ assuming $A$ is neither bad nor dodgy.}
\label{lead}
\end{figure}

We now construct a matching via a sequence of stages
$n=1,2,\ldots$.  At the end of stage $n$ we will have a partial
matching with each of its edges confined within some $n$-block.

\pagebreak
\paragraph{Stage $1$.}
Within each $1$-block, match as many red-blue pairs as possible
(for definiteness, choose from among the partial matchings of
maximal cardinality of the points in the block the one with
minimum total length).

\paragraph{Stage $n$ ($n\geq 2$).}
For each $n$-block $A$, let $B$ be the heir of $A$, and let
$C$ be the heir of $B$ (or let $C=B$ if $n=2$). Now:
\begin{mylist}
\item unmatch all points in $B$;
\item if possible, match all currently non-matched points in
    $A\setminus B$ to non-matched points in $(A\setminus B)\cup C$;
\item match as many of the remaining non-matched points in $A$
    as possible.
\end{mylist}
(In (ii) and (iii), for definiteness take the matching of
minimum length among those with the required property.  The
unmatching step (i) is a matter of convenience only - an
alternative would be to match only in blocks that lie in no
higher-order heir; see the finiteness claims below.)  We call
the block $A$ {\dof bad} if step (ii) does not succeed.  We
call a block {\dof dodgy} if at least one of its children is
bad. \parend\bigskip

Note that at the end of stage $n$, in any $n$-block $A$, the
number of non-matched points equals the excess
$|\red(A)-\blue(A)|$.  Moreover, if $n\geq 2$ and $A$ is not
bad, then all the non-matched points lie in the heir of $A$.
Furthermore (and this is the key observation), if $n\geq 3$ and
$A$ is neither bad nor dodgy then all the new edges added at
stage $n$ are confined to heirs (the heir $B$ of $A$ and the
heirs of the children of $A$).  See Figure \ref{lead}.

Let $Q:=[0,1)^2$ be the unit square; $Q$ is contained in
exactly one $n$-block for each $n$.  We will prove below that
a.s.\ $Q$ lies in only finitely many heirs, bad blocks and
dodgy blocks.  Once this is proved, the same also holds for
every integer unit square, and we deduce the following.  Each
point becomes unmatched (in step (i)) only finitely many times,
so we can define a limiting partial matching ${\cal M}$.  This
is in fact a perfect matching, since the only non-matched
points in a non-bad block lie in its heir.  Furthermore, from
the key observation above we deduce that $Q$ intersects only
finitely many edges a.s., so the same holds for any bounded set
as required.

Finally we turn to the proofs of the finiteness claims above.
Since the partitions into blocks are independent of
$\red,\blue$ we have
$$\P(\text{the $n$-block containing $Q$ is an heir})
=\frac{a_n a_{n-1}}{a_{n+1}a_n}=\frac{1}{n(n+1)}.$$
Since $\sum_n \frac{1}{n(n+1)}<\infty$, the Borel-Cantelli
lemma implies that $Q$ lies in only finitely many heirs a.s.

Also, the block $A$ is bad only if the net excess of one color
of points over the other in $A\setminus B$ cannot be
accommodated in $C$.  Thus, using Lemma \ref{poi},
\begin{align*}
\lefteqn{\P(Q\text{ lies in a bad $n$-block})} \\
{ }&\leq \P\big[(\red-\blue)(A\setminus B)> \blue(C)\big]+
\P\big[(\blue-\red)(A\setminus B)> \red(C)\big] \\
&\leq 2\exp -\frac{(a_{n-2} a_{n-3})^2}{6(a_n a_{n-1} - a_{n-1} a_{n-2})}
\leq 2\exp -\frac{n!^2}{6 n^9} \leq c_1e^{ -c_2 n},
\end{align*}
for some constants $c_i\in(0,\infty)$. Therefore
\begin{align*}
\P(Q\text{ lies in a dodgy $n$-block})\leq
\frac{a_n}{a_{n-2}} \,c_1 e^{ -c_2 (n-1)} \leq c_3 e^{-c_4 n}.
\end{align*}
Hence by the Borel-Cantelli lemma again, a.s.\ $Q$ lies
in only finitely many bad blocks and dodgy blocks.
\end{proof}

\section{The minimum matching}
%%%%%%%%%%%%%%%%%%%%%%%%%%%%%%

In this section we prove Theorem \ref{minimal}(i) in the case
$d\geq 3$.  The approach was suggested by Yuval Peres.

For a translation-invariant matching scheme $\mat$ of processes
$\red$ and $\blue$ both of intensity 1, define the average edge
length
$$\eta(\mat):=\frac 1{\leb S} \;\E \int_S |x-\mat(x)| \,d \red(x),$$
where $S\subset\R^d$ is any set with $\leb S\in(0,\infty)$.
The translation-invariance implies that $\eta(\mat)$ is
independent of the choice of $S$. (The quantity $\eta(\mat)$
also equals $\E |\mat^*(0)|$, where $\mat^*$ is the Palm
process obtained by conditioning on the presence of a red point
at the origin -- see e.g.\ \cite{hpps} or \cite[Ch.
11]{kallenberg}).

The key ingredient is the following fact, proved in
\cite[Theorem 1]{hpps}.
\begin{thm}[Mean edge length; \cite{hpps}]\label{finite-mean}
Let $\red$ and $\blue$ be independent Poisson processes of
intensity 1 in $\R^d$.  There exists a translation-invariant
matching scheme $\mat$ satisfying $\eta(\mat)<\infty$ if and
only if $d\geq 3$.
\end{thm}

\begin{cor}[Minimum matching]\label{min-mean}
Let $d\geq 3$ and let $\red$ and $\blue$ be independent Poisson
processes of intensity 1 in $\R^d$.  There exists a
translation-invariant matching scheme $\mhat$ such that
$$\eta(\mhat) =\min_{\mat} \eta(\mat),$$
\sloppy where the minimum is over all possible
translation-invariant matching schemes of $\red$ to $\blue$ (on
arbitrary probability spaces).
\end{cor}

Corollary \ref{min-mean} follows from Theorem \ref{finite-mean}
by an abstract argument, which we postpone to the end of the
section.

\begin{proof}[Proof of Theorem \ref{minimal}(i), case $d\geq 3$]
We claim that the matching scheme $\mhat$ from Corollary
\ref{min-mean} is minimal.  Suppose it is not.  Call a finite
set of edges {\dof reducible} if the incident red and blue
points can be rematched to give a strictly lower total length;
so with positive probability there exists a reducible set of
edges.  Therefore for some fixed $t$, with positive probability
there is some reducible set lying entirely in the cube
$[-t/2,t/2)^d$.

Construct a modified matching scheme $\mat'$ from $\mhat$ as
follows.  Within each of the disjoint cubes
$([0,t)^d+tz)_{z\in\Z^d}$, unmatch all the edges that lie
entirely within the cube, and replace them with the matching of
minimum length for this finite set of red and blue points.  The
resulting matching scheme satisfies the strict inequality
\begin{equation}\label{cube}
\frac 1{t^d} \;\E \int_{[0,t)^d} |x-\mat'(x)| \,d \red(x) < \eta(\mhat)
\end{equation}
(since the edges between the cube and its complement are
unaffected by the modification, while the total length of those
within it is never increased and sometimes decreased).  This
matching scheme $\mat'$ is not translation-invariant, but we
obtain a translation-invariant version $\mat''$ by translating
it by an independent uniform element of $[0,t)^d$.  The left
side of \eqref{cube} is unchanged if we replace $\mat'$ with
$\mat''$, so we obtain $\eta(\mat'')<\eta(\mhat)$,
contradicting Corollary \ref{min-mean}.
\end{proof}

\paragraph{Remarks.}
The same argument may be applied for instance to prove the
existence of a minimal one-color matching scheme for a Poisson
process (using \cite[Theorem 4]{hpps}).  Of course, the
approach cannot work for two-color matching in $\R^2$, since
there is no matching scheme with $\eta(\mat)<\infty$ (Theorem
\ref{finite-mean}).
\parend

\begin{proof}[Proof of Corollary \ref{min-mean}]
Recall that a matching scheme is a simple point process $\mat$
in $(\R^d)^2$, where the presence of an ordered pair $(r,b)\in
[\mat]$ signifies matched points $r\in[\red]$ and
$b\in[\blue]$. Note that we can recover the red and blue
processes from $\mat$ as
$\red_\mat(\cdot)=\red(\cdot):=\mat(\cdot\times\R^d)$, and
similarly for $\blue_\mat=\blue$.  In particular $\eta(\mat)$
is a function only of the law of $\mat$.

Let $I:=\inf_\mat \eta(\mat)$, where the infimum is over all
translation-invariant matching schemes of two independent
Poisson processes of intensity 1, and let
$\mat_1,\mat_2,\ldots$ be a sequence of such schemes such that
$$\eta(\mat_n)\searrow I \quad\text{as }n\to\infty.$$
We claim that the sequence $(\mat_n)$ is relatively compact in
distribution with respect to the vague topology on simple point
measures in $(\R^d)^2$.  This follows from \cite[Lemma
16.15]{kallenberg}, since any bounded set $A\subset (\R^d)^2$
is a subset of some $S\times \R^d$, where $S$ is Borel and
bounded, and $\mat_n(S\times \R^d)\stackrel{d}{=}\Poi(\leb S)$
for each $n$, thus $(\mat_n(S\times \R^d))$ is a tight
sequence. Therefore, by passing to a subsequence, we may assume
that for some simple point process $\mhat$ on $(\R^d)^2$,
$$\mat_n \stackrel{d}{\to} \mhat \quad\text{as }n\to\infty$$
in the aforementioned topology.  Since $I<\infty$, we may also
assume that
\begin{equation}\label{bound}
\eta(\mat_n)\leq C \quad\text{for all }n
\end{equation}
for some $C<\infty$. Clearly $\mhat$ is a matching scheme
between the point processes $\red=\red_\mhat$ and
$\blue=\blue_\mhat$. We will prove that it has all the required
properties.  The details will be largely routine, with the crucial
step being the use of the uniform bound \eqref{bound} to preclude
points being `matched to infinity' in the limit.

The above convergence implies that for any continuous,
compactly supported $f:(\R^d)^2\to[0,\infty)$ we have $\int f
d\mat_n\stackrel{d}{\to}\int f d\mhat$ \cite[Lemma 16.16(i)]{kallenberg}. We
first check that $\mhat$ inherits the translation-invariance of
$\mat_n$ -- this holds because for any such $f$ and its image
$f'$ under the diagonal action of some translation of $\R^d$ we
have $\int f d\mhat=\int f' d\mhat$. Next note that if
$D\subset\R^d$ is any closed, bounded, $\leb$-null set then
$\mhat(D\times\R^d)=0$ a.s., because we may choose
$\ind_{D\times\R^d} \leq f\leq\ind_{S\times\R^d}$ with $\leb S$
arbitrarily small, thus $\mhat(D\times\R^d)$ is stochastically
dominated by a $\Poi(\leb S)$ random variable. Therefore by
\cite[Lemma 16.16(iii)]{kallenberg} we have $\mat_n(S_1\times
S_2)\stackrel{d}{\to}\mhat(S_1\times S_2)$ for any bounded
Borel $S_1,S_2\subset\R^d$ with $\leb$-null boundaries.

Next we show that $\red=\red_\mhat$ is a Poisson process of
intensity 1.  It is enough to show that
$\mhat(S\times\R^d)\stackrel{d}{=}\Poi(\leb S)$ for any bounded
Borel $S$ with null boundary, but the problem is that
$S\times\R^d$ is not bounded. Suppose $S\subset B(t)$ and take
$T>t$.  We will approximate using $S\times B(T)$. We have for
any $T$,
$$\mat_n(S\times B(T))\stackrel{d}{\to} \mhat(S\times B(T))
 \quad\text{as }n\to\infty,$$
and also
$$\mhat(S\times B(T)) \stackrel{a.s.}{\to} \mhat(S\times \R^d) \quad\text{as
}T\to\infty.$$
We bound the approximation errors using Markov's inequality and
\eqref{bound}:
\begin{align*}
\P\Big[ \mat_n(S\times B(T)) \neq \mat_n(S\times \R^d &)\Big]
=\P\Big[\mat_n(S\times B(T)^c)> 0\Big] \\
&\leq \P\Big[\int_S|x-\mat_n(x)|\,d\red_{\mat_n}(x)\geq T-t\Big]\\
&\leq \frac{\eta(\mat_n) \,\leb S}{T-t} \leq \frac{C\, \leb S}{T-t},
\end{align*}
so this probability converges to $0$ as $T\to\infty$, uniformly
in $n$. (This uniformity is the key point of the proof).  By
\cite[Theorem 4.28]{kallenberg} it follows that
$\mat_n(S\times\R^d)\stackrel{d}{\to} \mhat(S\times\R^d)$, so
the latter has distribution $\Poi(\leb S)$ as required.

The same argument shows also that $\blue=\blue_\mhat$ is a
Poisson process of intensity 1 (here we use the fact that
$\eta(\mat_n)$ is equal to the analogous quantity with the
roles of red and blue reversed -- see e.g. \cite[Proposition
7]{hpps}).  We can prove that $\red$ and $\blue$ are
independent by applying similar reasoning to the joint law of
$\red(S_1)$ and $\blue(S_2)$ for bounded $S_1,S_2$.

We have established that $\mhat$ is a translation-invariant
matching scheme of two independent intensity-1 Poisson
processes, and it follows that $\eta(\mhat)\geq I$.  On the
other hand for any $\mat$ we have $\eta(\mat)=\sup_{k\to\infty}
\eta_k(\mat)$, where
$$\eta_k(\mat):=\E\int_{[0,1)^d} k \wedge |x-\mat(x)|  \;d\red_\mat(x),$$
and also $\eta_k(\mat_n)\to\eta_k(\mhat)$ for each $k$, so
$\eta(\mhat)\leq I$.  Thus $\eta(\mhat)= I$ as required.
\end{proof}

\section*{Acknowledgements}
I thank Yuval Peres, Oded Schramm and Terry Soo for many
valuable conversations.

\bibliography{bib}

\vspace{5mm}
\noindent {\sc\small
Microsoft Research, 1 Microsoft Way, Redmond, WA 98052, USA\\
and\\
Department of Mathematics, University of British Columbia,\\
121-1984 Mathematics Road, Vancouver, BC V6T 1Z2, Canada\\}
\url{holroyd(at)math.ubc.ca}  \ \
\url{www.math.ubc.ca/~holroyd}
\end{document}